\documentclass[12pt]{amsart}
\usepackage{amsmath,amsfonts,amscd,bezier,amsthm,amssymb}
\usepackage{mathrsfs,mathtools}
\usepackage{xcolor}
\usepackage{cite}
\usepackage{graphicx}
\usepackage{float}
\usepackage{comment}
\usepackage{enumitem,xfrac,nccmath}
\usepackage{tasks, multicol,yhmath}
\usepackage{hyperref}
\usepackage{tikz}
\usetikzlibrary{calc}

\usepackage[skip=10pt plus1pt, indent=15pt]{parskip}

\usepackage[pagewise]{lineno}

\newtheorem{theorem}{Theorem}[section]
\newtheorem{lemma}[theorem]{Lemma}
\newtheorem{corollary}[theorem]{Corollary}
\newtheorem{proposition}[theorem]{Proposition}
\theoremstyle{definition}

\newtheorem{remark}[theorem]{Remark}

\newtheorem{question}[theorem]{Question}

\newcommand{\R}{\mathbb{{R}}}

\newcommand{\T}{\mathbb{{T}}}

\newcommand{\D}{\mathbb{{D}}}

\newcommand{\rkhs}{\mathcal{H}}

\newcommand{\Hol}{\mathrm{Hol}(\D)}

\numberwithin{equation}{section}

\allowdisplaybreaks

\makeatletter
\DeclareRobustCommand{\eqrefp}[2]{%
	\textup{\tagform@{\refp{#1}{#2}}}}
\makeatother
\DeclareRobustCommand{\refp}[2]{%
	\expandafter\ifx\csname r@#1\endcsname\relax
	\textbf{??}%
	\else
	\edef\areferencia{\ref{#1}-)}%
	\expandafter\eqrefpaux\areferencia-#2%
	\fi}
\def\eqrefpaux#1-#2){#1}

\title[$A^p_\alpha$ classes in the Dirichlet range]{$A^p_\alpha$ classes in the Dirichlet range: inner-outer factorization, Carleson measures and weak products}

\author[A. Dayan]{Alberto Dayan}
	\address[A. Dayan]{Departament de Matemàtiques, Universitat Autònoma de Barcelona, 08193 Bellaterra (Barcelona), Spain} \email{alberto.dayan@uab.cat}

    \author[A. Llinares]{Adri\'{a}n Llinares}
    \address[A. Llinares]{Departamento de Matemáticas, Facultad de Ciencias, Universidad Autónoma de Madrid. C/ Francisco Tomás y Valiente 7, 28049 Madrid, Spain}
    \email{adrian.llinares@uam.es}

    \author[M. Monsalve-López]{Miguel Monsalve-López}
	\address[M. Monsalve-López]{Departamento de Análisis Matemático y Matem\'atica Aplicada, Facultad de Ciencias Matemáticas, Universidad Complutense de Madrid, Pl. de las Ciencias 3, 28040 Madrid, Spain}
	\email{migmonsa@ucm.es}

	\thanks{%
        The work of the three authors was partially supported by the grant PID2024-160326NA-I00 from the Spanish Ministry of Science and Innovation. The work of the first author was partially supported by the Emmy Noether Program of the German Research Foundation (DFG Grant 466012782) and by the grant PID2024-160185NB-I00 from the Spanish Ministry of Science and Innovation. The third author is also supported by the grant PID2022-137294NB-I00 from the Spanish Ministry of Science and Innovation}

	

	\date{\today}
    
\begin{document}
	
    \begin{abstract} 
        We study properties of $A^p_\alpha$ spaces in the Dirichlet range, recently defined by Brevig, Kulikov, Seip and Zlotnikov as the set of all holomorphic functions on the unit disc $\mathbb{D}$ such that
            \[
                \int_{\mathbb{D}} |f(z)|^{p-2} |f'(z)|^2 (1 - |z|^2)^{\alpha} \, dA(z) < \infty,
            \]
            when $0<\alpha < 1$ and $p > 0$. We answer in the negative two questions posed by Brevig et al. by showing that, if $p\ne2$ and $p > \frac{1}{2}$, $A^p_\alpha$ is not a vector space and that the norm is in general not increasing in $p$. This is achieved by means of an equivalent description for $A^p_\alpha$ which is given in terms of the Poisson integral of the boundary function of its inhabitants. Such norm also leads to a description of $A^p_\alpha$ functions in the Dirichlet range given in terms of their inner and outer factors. As a corollary, we show that $A^1_\alpha$ is contained in the weak product of a Dirichlet-type space.
	\end{abstract}
    
	\maketitle

    \section{Introduction}

        Let $p > 0$ be a positive exponent and $\alpha > 1$. For these values, the standard weighted Bergman space $A^p_\alpha$ is defined as the set of all $f \in \Hol$ such that the quantity
            \[
                \| f \|_{\alpha,p}^p := (\alpha - 1) \int_{\D} |f(z)|^p (1 - |z|^2)^{\alpha - 2} \, dA(z),
            \]
            where $dA(z)$ denotes the normalized Lebesgue area measure of $\D$, is finite. It is a well-known fact that, for every $p$ and every $f \in \Hol$,
            \[
                \lim_{\alpha \to 1} \| f \|_{\alpha,p}^p = \sup_{0 \leq r < 1} \dfrac{1}{2\pi} \int_0^{2\pi} |f(re^{it})|^p \, dt
            \]
            and thus it makes sense to understand the Hardy space $H^p$ as $A^p_1$. These spaces are very classical and have been extensively studied (see for instance the monographs \cite{MR268655, MR2033762, MR1758653, MR1669574}) but they have recently attracted a lot of attention because of the contractive embeddings between them and the preservation of conformal invariance on these chains of contractive inclusions. Namely, Kulikov showed in \cite{MR4472587} that
            \begin{equation} \label{eqn:ContractiveEmbbeding}
                \| f \|_{\beta,q} \leq \| f \|_{\alpha,p}
            \end{equation}
            for all $f \in A^p_\alpha$ as long as $p < q$, $1 \leq \alpha < \beta$ and $\frac{\alpha}{p} = \frac{\beta}{q}$. Furthermore, he also showed that equality in \eqref{eqn:ContractiveEmbbeding} is attained if and only if $f$ is a constant multiple of the function
            \[
                k_a (z) := \left( \frac{1 - |a|^2}{(1 - \overline{a}z)^2} \right)^\frac{\alpha}{p}, \quad a \in \D.
            \]
            In other words, $f$ must be a multiple of the image of the constant function $1$ by the weighted composition operator
            \begin{equation} \label{eqn:WComposition}
                T_a(g) = T_a^{\frac{\alpha}{p}} (g) :=\big( \varphi_a' \big)^\frac{\alpha}{p} \cdot (g \circ \varphi_a),
            \end{equation}
            where $\varphi_a (z) = \frac{a - z}{1 - \overline{a}z}$. Using a change of variables, it can be checked that $\| T_a \|$ is uniformly bounded (actually, constant) with respect to $a$ on $A^p_\alpha$ and therefore we say that the \emph{index of conformal invariance} of $A^p_\alpha$ is equal to $\frac{\alpha}{p}$ (see \cite{AM21}). If $p = 2$, then $A^2_\alpha$ is a reproducing kernel Hilbert space and the function $k_a$ is precisely the kernel at the point $a \in \D$.
            The second named author of this manuscript \cite{MR4705691} used the results of \cite{MR4472587} to prove that, when $p > 2$,
            \begin{equation} \label{eqn:Contractivep=2}
                \| f \|_{1, p} \leq \left( \sum_{n = 0}^\infty \dfrac{|a_n|^2}{c_{2/p} (n)} \right)^\frac{1}{2},
            \end{equation}
            where $c_{\alpha}(n) = \frac{\Gamma (n + \alpha)}{\Gamma(\alpha) n!}$, $\Gamma$ being the Gamma function, and $\{ a_n \}_{n \geq 0}$ is the sequence of Taylor coefficients of $f$. It should be noted that the right-hand side of \eqref{eqn:Contractivep=2} is finite if and only if $f' \in A^2_{2/p}$. In other words, if $f$ belongs to the Besov (or Dirichlet-type) space $B^2_{2/p}$, which is the unique Hilbert space that is conformally invariant with index $\frac{1}{p}$ (again, see \cite{AM21}).

        Since $M_p^p (r,f) := \frac{1}{2\pi} \int_{0}^{2\pi} |f(re^{it})|^p \, dt$ is a non-decreasing function with respect to $r$, one may use integration by parts to show that
            \begin{equation} \label{eqn:Apalphanorm}
                \| f \|_{\alpha, p}^p = |f(0)|^p + \int_0^1 \left(\dfrac{d}{dr} M_p^p (r,f) \right) (1 - r^2)^{\alpha - 1} \, dr, \quad \alpha \geq 1.
            \end{equation}
            Observe that $(1 - r^2)^{\alpha - 1}$ is also integrable in the range $0 < \alpha < 1$. This extends the class $A^p_\alpha$ to the range $\alpha \in (0,1)$ using \eqref{eqn:Apalphanorm} as an alternative expression for $\| f \|_{\alpha,p}$. As was shown in \cite{brevig2025contractivehardylittlewoodinequalitiesdirichlet}, $A^p_\alpha$ for $0 < \alpha < 1$ enjoys some of the properties that hold in the Hardy or weighted Bergman case. They showed that the operators defined in \eqref{eqn:WComposition} are isometries on $A^p_\alpha$ and that the inequality \eqref{eqn:ContractiveEmbbeding} holds for every $0 < p < q$ and $0 < \alpha < \beta$ with $\frac{\alpha}{p} = \frac{\beta}{q}$. In particular, this also provides an alternative proof for inequality \eqref{eqn:Contractivep=2} (which corresponds to the case $p = 2$ and $\beta = 1$). However, we must emphasize that $\| f \|_{\alpha, p}^p$ cannot be expressed as a $L^1$ norm of $|f|^p$ (or $|f'|^p$) with respect to a finite measure on $\overline{\D}$ when $\alpha \in (0,1)$. The case $p = 2$ (that is, the Hilbert case) is the unique exception of this fact. This will be of crucial importance for this work.
            
        An application of Hardy-Stein identity,
            \[
                \dfrac{d}{dr} M_p^p (r,f) = \dfrac{p^2}{2} \int_{r\D} |f(z)|^{p-2} |f'(z)|^2 \, dA(z),
            \]
            shows that 
            \begin{equation}
            \label{eqn:equivnorm}
            \|f\|_{\alpha, p}^p\simeq\int_{\D} |f(z)|^{p-2} |f'(z)|^2 (1 - |z|)^{\alpha} \, dA(z),
            \end{equation}
            with equivalence constants depending only on $p$ and $\alpha$.
            This provides yet another equivalent description of $A^p_\alpha$ for all $\alpha>0$. Note that in the range $p=2$ and $0<\alpha<1$ one recovers the  classical Dirichlet-type Hilbert spaces on the unit disc. Moreover, \eqref{eqn:equivnorm} shows that if $f$ does not vanish in $\D$ then
            \begin{equation}
                \label{eqn:nonvanishing}
                f\in A^p_\alpha\iff f^\frac{p}{2}\in A^2_\alpha.
            \end{equation}
            
        The main goal of this work is to deepen our understanding of the class $A^p_\alpha$. The following two questions were posed in \cite{brevig2025contractivehardylittlewoodinequalitiesdirichlet}:
        
        \begin{question}[Problem 1, \cite{brevig2025contractivehardylittlewoodinequalitiesdirichlet}]
        \label{q:linear}
            Let $0<\alpha<1$ and $p \neq 2$. Is $A^p_\alpha$ a vector space?
        \end{question}
        
        \begin{question}[Problem 2, \cite{brevig2025contractivehardylittlewoodinequalitiesdirichlet}]
        \label{q:increasing}
            Let $f$ be an analytic function on the unit disc and $0<\alpha<1$. Is $p\mapsto \|f\|_{\alpha, p}$ increasing?
        \end{question}
            
        Hölder's inequality yields that both questions above have positive answers when $\alpha\ge1$ (that is, when $A^p_\alpha$ is either a Hardy or a weighted Bergman space). We show that Question \ref{q:increasing} is quite far from having a positive answer in the range $0 < \alpha < \frac{1}{2}$, since in this range $A^p_\alpha\not\subset A^q_\alpha$ when $p > q > 0$ (see Corollary \ref{coro:notsub}). Actually, a direct argument in the whole range $0<\alpha<1$ yields that Question \ref{q:increasing} has a negative answer. Consider $e_n(z) = z^n$ for each integer $n \geq 1$. We have
        \[
            \| e_n \|_{\alpha, p} = \left(\int_0^1 \big(np r^{np-1} \big) (1 - r^2)^{\alpha - 1} \, dr \right)^{1/p}  = \left(\frac{\Gamma(\alpha)\Gamma\big(\frac{np}{2} + 1\big)}{\Gamma\big(\frac{np}{2} + \alpha \big)}\right)^{1/p}.
        \]
        Recall that $\Gamma$ is a log-convex (and therefore convex) function in $(0,+\infty)$. In particular, it can be checked that it is decreasing in $(0,1]$ and increasing in $[2, \infty)$. Thus, we may choose $n > \frac{4}{p}$ such that
        \[
            \Gamma(\alpha) > \Gamma(1) = 1 \quad \mbox{ and } \quad \Gamma \left(\frac{np}{2} + 1 \right) \geq \Gamma \left(\frac{np}{2} + \alpha \right),
        \]
        so $\| e_n \|_{\alpha, p} > 1$ but $\lim_{q \to \infty} \| e_n \|_{\alpha, q} = 1$. Hence, these quantities cannot be increasing with respect to the exponent $p$.

        We show that also Question \ref{q:linear} has a negative answer, at least for all $p > \frac{1}{2}$:

        \begin{theorem}
        \label{theo:NonLinear}
            Let $0<\alpha<1$ and $p \neq 2$. If $\alpha + p > \frac{1}{2}$, there exist two functions $f$ and $g$ in $A^p_\alpha$ such that $f+g$ is not in $A^p_\alpha$.
        \end{theorem}

        Note that $\| \cdot \|_{\alpha, p}$ is not a proper norm when $\alpha \in (0,1)$ and $p > \max \left \{ 0, \frac{1}{2} - \alpha \right\}$. Nevertheless, we will refer to it as a norm of $A^p_\alpha$ even when this class is not a linear space.

        Our main tool in the proof of Theorem \ref{theo:NonLinear} is an equivalent norm for the space $A^p_\alpha$ in the range $0<\alpha<1$, which is of its own interest. Note that $A^p_\alpha \subset  H^p$ if $\alpha<1$, hence every function $f$ in $A^p_\alpha$ has boundary values at almost every point of $\T := \partial \D$. For an integrable function $g$ on $\T$, let
            \[
            P_z(g):=\int_\T g(\xi)\frac{1-|z|^2}{|z-\xi|^2}~\frac{d\xi}{2\pi}
            \]
            denote its Poisson integral.
        
        \begin{theorem}
            \label{theo:equiv_A}
            Let $p>0$ and $0<\alpha<1$. A function $f$ in $H^p$ is in $A^p_\alpha$ if and only if 
                \[
                    \rho_{\alpha,p} (f) := \int_{\D}\left(P_z(|f|^p)-|f(z)|^p\right)(1-|z|^2)^{\alpha-2}\,dA(z)<\infty.
                \]
        \end{theorem}

        This result has to be understood as an analogue of equivalent norms for the Besov spaces deduced by Böe \cite{Boe03} and Dyakonov \cite{Dyakonov98}.
            
        Theorem \ref{theo:equiv_A} is used throughout this note to show that $A^p_\alpha$ spaces enjoy interesting properties even in the Dirichlet range, despite Theorem \ref{theo:NonLinear}. Suppose that $0<\alpha<1$. Since $A^p_\alpha\subset H^p$, any $f$ in $A^p_\alpha$ has an inner-outer factorization $f=\Theta h$.
        
        As an immediate corollary of Theorem \ref{theo:equiv_A}, one recovers the property that if $f$ is in $A^p_\alpha$ and $\Theta$ is an inner factor that divides $f$, then $f/\Theta$ is also in $A^p_\alpha$ (\cite[Theorem 1.8]{brevig2025contractivehardylittlewoodinequalitiesdirichlet}). Moreover, by using the inner-outer factorization in $A^p_\alpha$ and the equivalent norm from Theorem \ref{theo:equiv_A}, we provide a more exhaustive version of \eqref{eqn:nonvanishing} which does not require that $f$ is non-vanishing:
        
        \begin{theorem}
        \label{theo:p/2}
            Let $0<\alpha<1$, $p>0$, and $f=\Theta h$ a function in $H^p$, $\Theta$ and $h$ being its inner and outer factors, respectively. Then $f$ belongs to $A^p_\alpha$ if and only if $\Theta h^\frac{p}{2}$ is in $A^2_\alpha$. More generally, $f \in A^p_\alpha$ if and only if $\Theta h^{\frac{p}{q}} \in A^q_\alpha$, where $q$ is another positive exponent.
        \end{theorem}
        
        Note that, by setting $h=1$ in Theorem \ref{theo:p/2}, we can conclude that the membership of an inner function in $A^p_\alpha$ does not depend on $p$. Recall that all inner functions belong to $A^2_\alpha$ for $\alpha\ge1$, but this is not the case in the Dirichlet range. For instance, the atomic singular inner function belongs to $A^2_\alpha$ if and only if $\alpha> 1/2$, \cite{NS62}. Moreover, Theorem \ref{theo:p/2} and \eqref{eqn:nonvanishing} yield that a singular inner function $S_\mu$ is in $A^p_\alpha$ if and only if any of its powers is.

        Theorem \ref{theo:p/2} has consequences for the study of weak products of Dirichlet-type spaces, a description of which is, to this day, missing (see \cite[Question 4.7]{AHMR21}). Given a Hilbert space $\rkhs$ of functions on a domain, its weak product $\rkhs\odot\rkhs$ is the vector space of those functions that can be written as
            \begin{equation}
            \label{eqn:wp_rep}
                h=\sum_{i=1}^N f_ig_i, \qquad f_i, g_i\in\rkhs.
            \end{equation}
            The infimum of all values of $\sum_i\|f_i\|_\rkhs\|g_i\|_\rkhs$ among all the representations of $h$ in \eqref{eqn:wp_rep} is a norm on $\rkhs\odot\rkhs$. It is well known that $H^2\odot H^2=H^1$. Moreover, the fact that any function in $H^1$ is the product of two $H^2$ functions says that $N$ can be chosen to be equal to $1$ in \eqref{eqn:wp_rep} for all $h$ in $H^1$. This has been recently proved to be true for all complete Pick spaces, as a consequence of the column-row property, \cite[Theorem 1.4]{Hartz23}. As a corollary of Theorem \ref{theo:p/2}, we show that $A^1_\alpha$ is contained in the weak product of $A^2_\alpha$, thus providing a concrete function theoretical sufficient condition for the membership of a function in the weak product of a Dirichlet-type space:
            
        \begin{corollary}
        \label{coro:wp}
            Let $0<\alpha<1$. Then any $f$ in $A^1_\alpha$ is the product of two functions in $A^2_\alpha$. 
        \end{corollary}

        Note that Theorem \ref{theo:NonLinear} implies that the inclusion $A^1_\alpha\subset A^2_\alpha\odot A^2_\alpha$ must be proper.

        This manuscript is organized as follows. Section \ref{sec:equiv} is devoted to the proof of Theorem \ref{theo:equiv_A} and its consequences to the inner-outer factorization for $A^p_\alpha$ functions in the Dirichlet range, by proving Theorem \ref{theo:p/2} and Corollary \ref{coro:wp}. As a by-product of our arguments, we show that for $0<\alpha<1$ any function in $A^p_\alpha$ is the quotient of two bounded functions in $A^p_\alpha$, where the denominator is outer (see Theorem \ref{theo:quotient}). It is worth mentioning that this is already known for $p=2$, \cite{AHMR17}, since for $0<\alpha\le1$ the space $A^2_\alpha$ has the complete Pick property. In fact, in that case the denominator can be even chosen to be cyclic in $A^2_\alpha$, and both the numerator and the denominator are multipliers of $A^2_\alpha$.
        
        Section \ref{sec:NonLinear} uses the properties of $A^p_\alpha$ functions from Section \ref{sec:equiv} in order to prove Theorem \ref{theo:NonLinear}. This is achieved by adapting some techniques of Brown and Shields from \cite{BS84} in order to study the membership of a concrete class of functions in $A^p_\alpha$, see Proposition \ref{Prop:BS}.
        
        Finally, Section \ref{sec:cm} provides a proof, alternative to the one in \cite{brevig2025contractivehardylittlewoodinequalitiesdirichlet}, that the index of conformal invariance of $A^p_\alpha$ is $\frac{\alpha}{p}$ also for $0<\alpha<1$. Our proof uses Carleson measures for Dirichlet-type spaces and it can be extended to determine the index of conformal invariance of Besov spaces, which were recently found in \cite{AM21}, see Remark \ref{rem:AM}.

        \subsection*{Acknowledgments}

            We would like to express our gratitude to Ole Fredrik Brevig for his comments on the early drafts of this manuscript.

    \section{An Equivalent Norm on $A^p_\alpha$ in the Dirichlet range}
    \label{sec:equiv}

        Theorem \ref{theo:equiv_A} is a consequence of sub-harmonicity. If $u$ is sub-harmonic and $u\in C^2(\overline\D)$, Green's formula yields
            \begin{equation}
            \label{eqn:Poisson_sub}
                P_{z}(u)-u(z)=\int_\D G_z(w)\Delta u(w)\,dA(w),
            \end{equation}
            where 
            \[
                G_z(w)\coloneqq\log\left|\frac{1-\overline{z}w}{w-z}\right|, \qquad z, w\in\D
            \]
            is the Green function of the unit disc with pole at $w$. Recall that for any analytic function $f$, then $|f|^p$ is sub-harmonic for all $p>0$.

        \begin{proof}[Proof of Theorem \ref{theo:equiv_A}]
            We apply \eqref{eqn:Poisson_sub} with $u=|f|^p$. Since $\Delta(|f|^p)=p^2 |f'|^2|f|^{p-2}$, we obtain
                \[
                    \begin{split}
                        &\int_\D\left(P_z(|f|^p)-|f(z)|^p\right)(1-|z|^2)^{\alpha-2}\,dA(z)\\
                        =&\int_\D \left(\int_\D G_z(w)\Delta(|f|^p)(w)\,dA(w) \right)(1-|z|^2)^{\alpha-2}\,dA(z)\\
                        \simeq&\int_\D|f'(w)|^2|f(w)|^{p-2} \left(\int_\D G_w(z)(1-|z|^2)^{\alpha-2}\,dA(z) \right) dA(w).
                    \end{split}
                \]
                We are therefore left to show that
                \begin{equation}
                \label{eqn:comparableweights}
                    \int_\D G_w(z)(1-|z|^2)^{\alpha-2}\,dA(z)\simeq (1-|w|)^\alpha.
                \end{equation}
                Using a conformal change of variables, the left side of \eqref{eqn:comparableweights} can be rewritten as
                \[
                    \int_\D \log \left( \dfrac{1}{|z|}\right) (1 - |\varphi_w (z)|^2)^\alpha \dfrac{dA(z)}{(1 - |z|^2)^2} = (1 - |w|^2)^\alpha \int_\D \log \left( \dfrac{1}{|z|}\right)  \dfrac{(1 - |z|^2)^{\alpha - 2}}{|1 - \overline{w} z|^{2\alpha}} dA(z).
                \]
                Now using that $r \mapsto \log\big(\tfrac{1}{r}\big)/(1-r^2)$ is a decreasing function in $(0,1)$, Chebyshev's inequality (see \cite{MR1505400}) yields that the quantities
                \[
                   \int_\D |f(z)|^p (1 - |z|^2)^{\alpha - 2} \log \left( \frac{1}{|z|} \right) \, dA(z) 
                \]
                and
                \[
                    \int_\D |f(z)|^p (1 - |z|^2)^{\alpha - 1} \, dA(z)
                \]
                are comparable for every $f \in \Hol$. Therefore, we are left to observe that the functions
                \[
                    f_w(z)=\frac{1}{(1-\overline wz)^{\alpha}}
                \]
                have bounded norms in $A^2_{\alpha+1}$, uniformly in $w$, if $0<\alpha<1$. This holds because of the fact that $f_1 \in A^2_{\alpha + 1}$ if and only if $\alpha < \frac{\alpha + 1}{2}$. That is, when $\alpha < 1$.
        \end{proof}

        For all $p>0$ and $\alpha<1$, any function $f$ in $A^p_a$ is also in $H^p$. In particular, any such $f$ can be written as $f=\Theta h$, where $\Theta$ is inner and $h$ is outer.
 
        \begin{corollary}
        \label{coro:inner_outer}
            Let $0<\alpha<1$ and $p>0$. Given any function $f$ in $A^p_\alpha$ and any inner function $\Theta$, the product $\Theta f$ belongs to $A^p_\alpha$ if and only if
            \[
                \int_{\D}(1-|\Theta(z)|)|f(z)|^p(1-|z|^2)^{\alpha-2}\,dA(z)<\infty.
            \]
        \end{corollary}
        
        \begin{proof}
            This follows from Theorem \ref{theo:equiv_A} and the fact that 
            \[
                P_z(|\Theta f|^p)-|\Theta(z)f(z)|^p=P_z(|f|^p)-|f(z)|^p+(1-|\Theta(z)|^p)|f(z)|^p
            \]
            for all $z$ in $\D$.
        \end{proof}

        We are ready to prove Theorem \ref{theo:p/2}:
        
        \begin{proof}[Proof of Theorem \ref{theo:p/2}]
            If $\Theta$ is constant then $f$ does not vanish, and the results follows directly from \eqref{eqn:nonvanishing}.
        
            If $f=\Theta h$ is in $A^p_\alpha$, then $h$ is in $A^p_\alpha$ thanks to Theorem \ref{theo:equiv_A}. Therefore $h^\frac{p}{2}$ is in $A^2_{\alpha}$, by \eqref{eqn:nonvanishing}. Corollary \ref{coro:inner_outer} yields that $\Theta h^\frac{p}{2}$ is in $A^2_{\alpha}$. The same argument yields the other direction of the claim.
        \end{proof}

        This yields Corollary \ref{coro:wp}:
        
        \begin{proof}[Proof of Corollary \ref{coro:wp}]
            Let $f=\Theta h$ be in $A^1_\alpha$. By Theorem \ref{theo:p/2}, $\Theta\sqrt h$ is in $A^2_\alpha$, and so is $\sqrt{h}$ by Theorem \ref{theo:equiv_A}. Since $f=(\Theta\sqrt h)\sqrt h$, the claim follows.
        \end{proof}

        Finally, we use the inner-outer factorization of a function $f$ in $A^p_\alpha$ to show that any such function is the quotient of two bounded functions in $A^p_\alpha$. To see this, note that, given an outer function $h$, one can obtain two truncated outer functions as follows: 
            \[
                h_{\mathrm{min}}(z):=\exp\left\{\int_\T\frac{\xi+z}{\xi-z}\log\min\{1, |h(\xi)|\}\,d\xi\right\}
            \]
            and 
            \[
                h_{\mathrm{max}}(z):=\exp\left\{\int_\T\frac{\xi+z}{\xi-z}\log\max\{1, |h(\xi)|\}\,d\xi\right\}.
            \]
            In particular, $h_{\mathrm{min}}$ is bounded by $1$, $h_{\mathrm{max}}$ is bounded below by $1$ and $h=h_{\mathrm{min}}h_{\mathrm{max}}$. This yields that any function $f=\Theta h$ in $A^p_\alpha$ can be written as 
            \begin{equation}
            \label{eqn:quotient}
                f(z)=\frac{\Theta(z)h_{\mathrm{min}}(z)}{h^{-1}_\mathrm{max}(z)}, \qquad z\in\D,
            \end{equation}
            where both the numerator and the denominator in \eqref{eqn:quotient} are bounded.
            
        \begin{theorem}
            \label{theo:quotient}
            Let $p>0$, $0<\alpha<1$ and let $f=\Theta h$ be a function in $A^p_\alpha$. Then $\Theta h_\mathrm{min}$,  $h_\mathrm{max}$ and $h^{-1}_\mathrm{max}$ are in $A^p_\alpha$. In particular, any function $f$ in $A^p_\alpha$ can be written as the quotient of two bounded functions in $A^p_\alpha$, where the denominator is outer.
        \end{theorem}
        
        \begin{proof}
            For $p=2$, this is proved in \cite[Thm. 3.3, 3.4]{Boe03}. For other values of $p$, it follows from the case $p=2$, from Theorem \ref{theo:p/2} and the fact that 
            \[
                h_\mathrm{min}^\frac{p}{2}=\left(h_\mathrm{min}\right)^\frac{p}{2}\quad \mbox{ and } \quad h_\mathrm{max}^\frac{p}{2}=\left(h_\mathrm{max}\right)^\frac{p}{2}.
            \]
        \end{proof}

    \section{Membership of a concrete class of functions in $A^p_\alpha$}
    \label{sec:NonLinear}

        We adapt a construction from \cite[Prop. 16]{BS84} to weighted Dirichlet spaces. We include the proof for completeness, since we will need a more quantitative statement (see Remark \ref{rem:power_est} below).
    
        \begin{proposition} \label{Prop:BS}
            Let $a \ge 0$, $b \in \R$ and $0<\alpha<1$. Consider the function
                \[
                    f_{a, b}(z) := (1-z)^b\exp\left(-a\frac{1+z}{1-z}\right).
                \]
                Then:
                \begin{enumerate}
                    \item \label{Propi)} $f_{a, b}$ belongs to $A_\alpha^2$ if $b > \frac{1}{2} - \alpha$.
                    
                    \item \label{Propii)} If $b>-a$ and $b\le\frac{1}{2}-\alpha$, then $f$ does not belong to $A^p_\alpha$.
                \end{enumerate}
        \end{proposition}
    
        \begin{proof}
    
           We start with \ref{Propi)}. Note that
                \[
                    f_{a,b}'(z) = -\big(b(1-z) + 2a\big) (1-z)^{b-2}\exp\left(-a \frac{1+z}{1-z}\right).
                \]
                Therefore,
                \[
                    \|f_{a,b}\|_{\alpha,2}^2 \lesssim (a+|b|)^2 \int_{\mathbb{D}} |1-z|^{2(b-2)} \bigg|\exp\bigg(-a\frac{1+z}{1-z}\bigg)\bigg|^{2} (1-|z|^2)^{\alpha} \, dA(z).
                \]
                
            Mapping the disk $\D$ to the open right half-plane $\mathbb{H}$ with the change of variable
                \[
                    w = \frac{1+z}{1-z} \quad \mbox{ and } \quad \phi(w) = \frac{w - 1}{w +1} = z
                \]
                and writing $w = u + i v$, we obtain
                \[
                \begin{aligned}
                    \|f_{a,b}\|_{\alpha,2}^2 & \lesssim (a+|b|)^2 \int_\mathbb{H} |1-\phi(w)|^{2(b-2)} |e^{-aw}|^{2} (1-|\phi(w)|^2)^{\alpha} |\phi'(w)|^2\, du dv \\
                    & \simeq (a + |b|)^2 4^b \int_0^{+\infty} u^{\alpha} e^{-2au} \Bigg(\int_{-\infty}^{+\infty} \big(|1+w|^2\big)^{-b-\alpha} dv\Bigg) du.
                \end{aligned}
                \]
        
            Using a standard change of variables, we can check that
                \[
                    \int_{-\infty}^{+\infty} \big(|1+w|^2\big)^{-b-\alpha} dv  = \dfrac{2}{(u+1)^{2(b+\alpha)-1}} \int_0^{\frac{\pi}{2}} (\cos t)^{2(b+\alpha - 1/2) - 1} \, dt = \dfrac{B \left( b + \alpha - \frac{1}{2}, \frac{1}{2} \right)}{(u + 1)^{b + \alpha}},
                \]
                where $B$ is Euler's Beta function. Thus, if $b > \frac{1}{2} - \alpha$ this integral is finite and therefore
                \[
                    \|f_{a,b}\|_{\alpha,2}^2 \lesssim (a + |b|)^2 4^b B \left( b + \alpha - \frac{1}{2}, \frac{1}{2} \right) \int_0^{+\infty} \dfrac{u^\alpha}{(u + 1)^{2(b+\alpha) - 1}} e^{-2au} du < +\infty.
                \]
    
            As for \ref{Propii)}, if $b > -a$, the same argument shows that the inequality
            \[
                \|f_{a,b}\|_{\alpha,2}^2 \gtrsim (\min\{a,a+b\})^2 4^b B \left( b + \alpha - \frac{1}{2}, \frac{1}{2} \right) \int_0^{+\infty} \dfrac{u^\alpha}{(u + 1)^{2(b+\alpha) - 1}} e^{-2au} du
            \]
            holds.
        \end{proof}

        \begin{remark}
        \label{rem:fab}
            Thanks to Theorem \ref{theo:p/2}, for all $p>0$ and $0<\alpha<1$ we have that $f_{a, b}$ belongs to $A^p_\alpha$ if and only if $f_{a, pb/2}$ is in $A^2_\alpha$. Hence if $b>-2a/p$, one has
            \[
                f\in A^p_\alpha\iff b>\frac{1}{p}-\frac{2\alpha}{p}.
            \]
        \end{remark}
       
        \begin{corollary}
        \label{coro:notsub}
            If $0 < \alpha <\tfrac{1}{2}$ and $p > q>0$, then $A^p_\alpha$ is not contained in $A^q_\alpha$.
        \end{corollary}

        \begin{proof}
Fix $b:=\frac{1}{q}(1-2\alpha)$.  Since $1-2\alpha>0$, we have that
\[
bp>bq=1-2\alpha.
\]
Hence Remark \ref{rem:fab} yields that $f_{1, b}$ belongs to $A^p_\alpha$ but not to $A^q_\alpha$.

        \end{proof}

        \begin{remark} \label{rem:power_est}
            Using a change of variables and the monotone convergence theorem, we can show that
                \begin{align*}
                    \lim_{j \to \infty} j^{1 + \alpha} \int_0^{+\infty} \dfrac{u^\alpha}{(u + 1)^{2(jb + \delta + \alpha) - 1}} e^{-2jau} du & = \lim_{j \to \infty} \dfrac{1}{(2a)^{1 + \alpha}} \int_0^{+\infty} \dfrac{t^{\alpha}}{ \left( \frac{t}{2ja} + 1 \right)^{2(jb+\delta + \alpha)-1}} e^{-t} \, dt \\
                        & = \dfrac{1}{(2a)^{1+\alpha}}\int_0^{+\infty} t^\alpha e^{-t (1 + b/a)} \, dt = \dfrac{\Gamma (\alpha + 1)}{\big(2 (a+b) \big)^{1+\alpha}}
                \end{align*}
                if $a + b > 0$. Thus we have that
                \[
                    \|f_{ja, jb+\delta}\|_{\alpha,2} \simeq j^{\frac{1 - \alpha}{2}} 2^{jb + \delta} \sqrt{B \left( jb + \delta + \alpha - \dfrac{1}{2}, \dfrac{1}{2} \right)} \simeq j^{\frac{1}{4} - \frac{\alpha}{2}} 2^{jb}.
                \]
                Since $\left|\binom{\frac{p}{2}}{j}\right|\simeq \frac{1}{j^{1+\frac{p}{2}}}$ when $j \to \infty$, we will have that
                \[
                    \left|\binom{\frac{p}{2}}{j}\right| \dfrac{\|f_{ja, jb+\delta}\|_{\alpha,2}}{2^{jb}} \simeq j^{-\frac{3}{4} - \frac{\alpha + p}{2}}
                \]
                which is summable when $\alpha + p > \frac{1}{2}$.
        \end{remark}
    
        Our counterexample for the proof of Theorem \ref{theo:NonLinear} is also built upon Proposition \ref{Prop:BS}:

        \begin{proof}[Proof of Theorem \ref{theo:NonLinear}]
        	Consider the function $f_{1, b}$ as in Proposition \ref{Prop:BS} and 
                \[
                    g_c(z):=2^{b-c}(1-z)^c,
                \]
                where $0<\varepsilon<1-\alpha$ and
                \[
                    c=\frac{\varepsilon-\alpha}{p}.
                \]
                A standard computation using Theorem \ref{theo:p/2} and the Taylor series of $(1-z)^\frac{cp}{2}$ shows that $(1-z)^c$ belongs to $A^p_\alpha$, since $c>-\frac{\alpha}{p}$. On the other hand, Proposition \ref{Prop:BS} yields that $f_{1, b}$ is in $A^p_\alpha$ when 
                \begin{equation}
                \label{eqn:f_1b}
                    b>\frac{1}{p}(1-2\alpha).
                \end{equation}
            
            Let $h(z) := f_{1, b}(z) + g_c(z)$. If $b$ satisfies \eqref{eqn:f_1b}, then $b - c > \frac{1 - \alpha - \varepsilon}{p} > 0$ and hence 
                \[
                    \left| \dfrac{1-z}{2} \right|^{b-c} \left|\exp\left(-\frac{1+z}{1-z}\right)\right| < 1, \quad z \in \mathbb{D}.
                \]
                Thus, the function $h$ is nowhere vanishing in $\D$. Thanks to \eqref{eqn:nonvanishing}, $h$ lies into $A_\alpha^p$ if and only if $h^{p/2}$ belongs to $A^2_\alpha$. Moreover,
            	\[
            		\begin{aligned}
            		\big(h(z)\big)^{p/2} & = 2^{\frac{p(b-c)}{2}} (1-z)^{\frac{cp}{2}}\Bigg(1+ 2^{-(b-c)}(1-z)^{b-c}\exp\bigg(-\frac{1+z}{1-z}\bigg)\Bigg)^{p/2} \\
            		& = 2^{\frac{p(b-c)}{2}} (1-z)^{\frac{cp}{2}} \sum_{j = 0}^\infty \binom{p/2}{j} \frac{\big(f_{1,b-c}(z)\big)^j}{2^{j(b-c)}} \\
            		& = 2^{\frac{p(b-c)}{2}}  \sum_{j = 0}^\infty \binom{p/2}{j} \frac{f_{j,j(b-c) + \frac{cp}{2}}(z)}{2^{j(b-c)}}.
            		\end{aligned}
            	\]
            
            Since $ A^2_\alpha$ is a Hilbert space,
                \[
                    \left\|\big(h(z)\big)^{p/2}\right\|_{\alpha, 2}\le\sum_{j = 0}^\infty \left|\binom{p/2}{j}\right| \frac{\|f_{j,j(b-c) + \frac{cp}{2}}\|_{\alpha, 2}}{2^{j(b-c)}},
                \]
                thus Proposition \ref{Prop:BS} and Remark \ref{rem:power_est} yield that $h \in A^p_\alpha$ if 
                    \begin{equation}
                    \label{eqn:h}
                        b-c+\frac{cp}{2}>\frac{1}{2}-\alpha.
                    \end{equation}
            
            On the other hand, if 
                \begin{equation}
                \label{eqn:cnot}
                    b-c+\frac{cp}{2}=\frac{1}{2}-\alpha.
                \end{equation}
                then
                \[
                    \left\|\big(h(z)\big)^{p/2}\right\|_{\alpha, 2}\ge\left| \frac{p}{2^{b-c+1}}\|f_{1, b-c+\frac{cp}{2}}\|_{\alpha, 2}- \left \| \sum_{j \ne1} \binom{p/2}{j} \frac{f_{j,j(b-c) + \frac{cp}{2}}}{2^{j(b-c)}} \right \|_{\alpha, 2} \right|=\infty,
                \]
                thus $h$ is not in $A^p_\alpha$ in this case.
        
            The argument now distinguishes two cases:
                \begin{itemize}
                
                    \item Suppose that $p>2$. We will choose $b$ such that $f_{1, b}$ is in $A^p_\alpha$ but $h$ is not. To this end, fix $b$ as in \eqref{eqn:cnot}. We are left to check that $f_{1, b}$ is in $A^p_\alpha$: this is the case because inequality \eqref{eqn:f_1b} is equivalent to 
                     \[
                     \varepsilon\left(\frac{1}{p}-\frac{1}{2}\right)>(1-\alpha)\left(\frac{1}{p}-\frac{1}{2}\right),
                     \]
                    which holds since $\varepsilon<1-\alpha$ and $p>2$.
                     
                    \item Now assume that $p<2$, and fix 
                    \[
                    b=\frac{1}{p}(1-2\alpha),
                    \]
                    so that $f_{1, b}$ is not in $A^p_\alpha$. In order to check that $h$ is in $A^p_\alpha$ we are left to show that \eqref{eqn:h} holds. A computation shows that, with the parameters $b$ and $c$ that we fixed, \eqref{eqn:h} is equivalent to 
                    \[
                    \varepsilon\left(\frac{1}{2}-\frac{1}{p}\right)>\left(\frac{1}{2}-\frac{1}{p}\right)(1-\alpha)
                    \]
                    which holds since $\varepsilon<1-\alpha$ and $p<2$.
                \end{itemize}
        \end{proof}

        \begin{remark} \label{rem:RemaingCases}
    
            Observe the following:
            \begin{enumerate}
                
                \item If the series
                    \[
                        \sum_{j = 2}^\infty \binom{p/2}{j} \dfrac{f_{j,j(b-c)+\frac{cp}{2}} (z)}{2^{j(b-c)}}
                    \]
                    was convergent but not absolutely convergent in $A^2_\alpha$, our argument would show that $A^p_\alpha$ is not a linear space in the whole range $p, \alpha > 0$.
        
                \item If $p > 2$ and $0< \alpha < \frac{1}{2}$, then we can pick $c = 0$ in the argument above. This says that for all $0<\alpha<\tfrac{1}{2}$ there exists a function $f\in A^p_\alpha$ and a constant $C$ such that $f+C$ is not in $A^p_\alpha$.
            \end{enumerate}
        	
        \end{remark}

    \section{A class of Carleson measures arising from conformal invariance}
    \label{sec:cm}

        Corollary \ref{coro:inner_outer} yields a description of Carleson measures for $A^p_\alpha$, that is, of those positive Borel measures on the unit disc for which the embedding
            \[
                A^p_\alpha\subset L^p(\D, d\mu)
            \]
            is bounded. 

        \begin{proposition}
        \label{prop:CM_same}
            Let $0<\alpha<1$ and $p, q>0$. Then, a positive Borel measure $\mu$ is Carleson for $A^p_\alpha$ if and only if it is Carleson for $A^q_\alpha$. 
        \end{proposition}
    
        \begin{proof}
            Thanks to Theorem \ref{theo:p/2} a function $f = \Theta h$, where $\Theta$ is inner and $h$ is outer, belongs to $A^p_\alpha$ if and only if $\Theta h^\frac{p}{q}\in A^q_\alpha$. Hence, if $\mu$ is Carleson for $A^q_\alpha$, then for all $f=\Theta h$ in $A^p_\alpha$ we have that
                \[
                    \|f\|_{L^p(d\mu)}^p \le \|h\|_{L^p(d\mu)}^p = \|h^\frac{p}{q}\|_{L^2(d\mu)}^q\lesssim\|h^\frac{p}{q}\|_{\alpha, q}^q\lesssim\|\Theta h^\frac{p}{q}\|_{\alpha, q}^q=\|f\|^p_{\alpha, p},
                \]
                where the last inequality uses the alternative description of $A^q_\alpha$ provided by Theorem \ref{theo:equiv_A}. Thus $\mu$ is Carleson for $A^p_\alpha$.
                
            Changing the role of $p$ and $q$, one obtains the converse implication.
        \end{proof}

        In particular, in order to get a description of Carleson measures for the spaces $A^p_\alpha$ it suffices to describe Carleson measures for Dirichlet-type spaces. Such measures are well understood: Stegenga \cite{Stegenga80} and Wu \cite{Wu99} describe them in terms of capacity conditions for union of boxes in the unit disc. An alternative characterization was provided by Arcozzi, Rochberg and Sawyer in \cite[Thm. 1]{ARS02}. The advantage of the latter characterization is that no capacity is involved. Given a point $w$ in $\D$, its associated Carleson box is defined as
            \[
                S_w:=\left\{z\in\D\,\bigg|\, 1-|z|\leq 1-|w|,\,\, \frac{|\arg(z\overline w)|}{2\pi}<1-|w|\right\},
            \]
            while its stretched Carleson box is defined as
                \[
                    \tilde S_w:=\left\{z\in\D\,\bigg|\, 1-|z|\leq 2(1-|w|),\,\, \frac{|\arg(z\overline w)|}{2\pi}<1-|w|\right\}.
                \]

        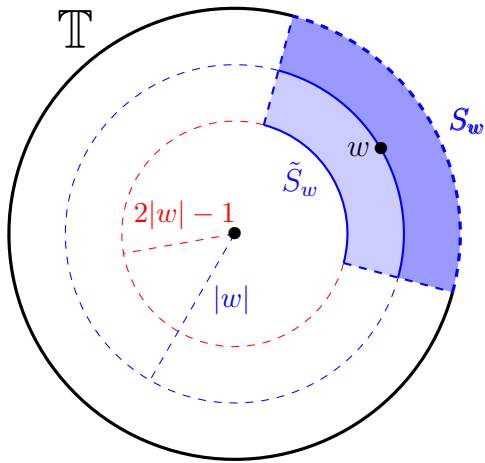
\begin{figure}[h!]
            \begin{tikzpicture}[x=3cm,y=3cm]
                \def \mod {0.75};
                \def \arg {30};
                \pgfmathsetmacro{\stretch}{2*\mod-1}
                \pgfmathsetmacro{\angle}{180*(1-\mod)}
                
                \draw[very thick, black] ({\arg + \angle}:1) arc({\arg + \angle}:{\arg - \angle + 360}:1);
                \draw[very thick, dashed, blue] ({\arg - \angle)}:1) arc({\arg - \angle}:{\arg + \angle}:1);
                \draw[dashed, blue] ({\arg + \angle}:\mod) arc({\arg + \angle}:{\arg - \angle + 360}:\mod);
                \draw[thick, blue] ({\arg - \angle)}:\mod) arc({\arg - \angle}:{\arg + \angle}:\mod);
                \draw[dashed, red] ({\arg + \angle}:{\stretch}) arc({\arg + \angle}:{\arg - \angle +360}:{\stretch});
                \draw[thick, blue] ({\arg - \angle}:{\stretch}) arc({\arg - \angle}:{\arg + \angle}:{\stretch});
                \draw[dashed, red] (0,0) -- ($ \stretch*({cos(\arg + 160)},{sin(\arg + 160))}) $) node[pos=0.45, above=1pt]{\small $2|w|-1$};
                \draw[dashed, blue] (0,0) -- ($ \mod*({cos(\arg + 210)},{sin(\arg + 210))}) $) node[pos=0.45, right=2pt]{\small $|w|$};
                \draw (0,0) node{$\bullet$};
                \draw ($ ({cos(135)},{sin(135)}) $) node[above=6pt]{\LARGE $\mathbb{T}$};
                \draw[thick, dashed, blue] ($ \stretch*({cos(\arg + \angle)},{sin(\arg + \angle)}) $) -- ($ ({cos(\arg + \angle)},{sin(\arg + \angle)}) $);
                \draw[thick, dashed, blue] ($ \stretch*({cos(\arg - \angle)},{sin(\arg - \angle)}) $) -- ($ ({cos(\arg - \angle)},{sin(\arg - \angle)}) $);
                \fill[blue, opacity=0.2] ({\arg - \angle}:\mod) arc({\arg - \angle}:{\arg + \angle}:\mod)--({\arg + \angle}:{\stretch}) arc({\arg + \angle}:{\arg - \angle}:{\stretch})--cycle;
                \fill[blue, opacity=0.4] ({\arg - \angle}:\mod) arc({\arg - \angle}:{\arg + \angle}:\mod)--({\arg + \angle}:{1}) arc({\arg + \angle}:{\arg - \angle}:{1})--cycle;
                \draw[blue] ($ ({cos(\arg)},{sin(\arg)}) $) node[right=3pt]{$S_w$};
                \draw[black] ($ \mod*({cos(\arg)},{sin(\arg)}) $) node{$\bullet$} node[left]{$w$};
                \draw[blue] ($ ({cos(\arg)},{sin(\arg)}) $) node[right=3pt]{\small $S_w$};
                \draw[black] ($ \mod*0.65*({cos(\arg)},{sin(\arg)}) $) node[left, blue]{$\tilde{S}_w$};
            \end{tikzpicture}
            \caption{Visualization of the Carleson boxes $S_w$ and $\tilde{S}_w$.}
        \end{figure}

        \begin{theorem}[Arcozzi-Rochberg-Sawyer \cite{ARS02}]
        \label{theo:ARS}
            Let $\mu$ a positive Borel measure on $\D$. Then $\mu$ is a Carleson measure for $ A^2_\alpha$ if and only if there exists a constant $C>0$ such that
                \begin{equation}
                \label{eqn:ARS}
                    \int_{\tilde S_w}\mu(S_z\cap S_w)^2\frac{dA(z)}{(1-|z|)^{2+\alpha}}\leq C\mu(S_w)
                \end{equation}
                for all $w\in\D$. If so, the norm of the embedding $ A^2_\alpha$ depends only on $C$.
        \end{theorem}

        The goal of this section is to relate Carleson measures for $A^p_\alpha$ with their index of conformal invariance:
        
        \begin{theorem}[\cite{brevig2025contractivehardylittlewoodinequalitiesdirichlet}]
        \label{theo:confinv}
            The index of conformal invariance of $A^p_\alpha$ is $\frac{\alpha}{p}$ for all $\alpha, p>0$.
        \end{theorem}
        
        We provide an alternative proof of Theorem \ref{theo:confinv} by using Carleson measures.
        The bridge between Carleson measures and conformal invariance is provided by the following proposition. For all positive real numbers $\gamma$ and $s$, consider the measure 
            \begin{equation}\label{eqn:mu}
                d\mu^{\gamma, s}_a (z) := \frac{(1-|z|)^\gamma}{|1-\overline az|^s}~dA(z),\qquad a\in\D.
            \end{equation}
        
        \begin{proposition}
        \label{prop:conf_index_cm}
            Let $\alpha, p>0$. The weighted composition operators defined by 
            \[
                T_af(z):=\left(\varphi_a'(z)\right)^\frac{\alpha}{p}f(\varphi_a(z)), \quad z\in\D,
            \]
            satisfy
                \[
                    \sup_{a\in\D}\|T_a\|_{\mathcal{B}(A^p_\alpha)}<\infty
                \]
                if and only if the measures $\{ \mu^{\alpha, 2}_a \}_{a\in\D}$ (following \eqref{eqn:mu}) are Carleson for $A^p_{\alpha}$, uniformly in $a$.
        \end{proposition}
    
        \begin{proof}
        
            Recall that
                \[
                    \|T_af\|_{\alpha,p}^p \simeq \int_{\D} \big|T_af(z)\big|^{p-2} \big|(T_af)'(z)\big|^2 (1-|z|^2)^{\alpha} ~dA(z).
                \]
                Since
                \[
                    (T_a f)' = \dfrac{\alpha}{p} \big( \varphi_a' \big)^{\frac{\alpha}{p} - 1} \varphi_a'' f \big( \varphi_a ) + \big( \varphi_a' \big)^{\frac{\alpha}{p} + 1} f' (\varphi_a)
                \] 
                we have that $$|I_1-I_2|\lesssim\|T_af\|_{\alpha,p}^p \lesssim I_1 + I_2,$$ where:
                \[
                    \begin{aligned}
                        I_1 & \coloneqq \int_{\D} \big|f(\varphi_a(z))\big|^{p-2}\big|f'(\varphi_a(z))\big|^{2}\big|\varphi_a'(z)\big|^{\alpha+2} (1-|z|^2)^{\alpha}~dA(z),\\
                        I_2 & \coloneqq \int_{\D} \big|f(\varphi_a(z))\big|^{p}\big|\varphi_a'(z)\big|^{\alpha - 2}\big| \varphi_a''(z)\big|^{2} (1-|z|^2)^{\alpha}~dA(z).
                    \end{aligned}
                \]
                Thanks to conformal invariance, the integral $I_1$ is comparable to $\|f\|^p_{\alpha,p}$, independently on $a$:
                \[
                    \begin{aligned}
                        I_1 & = \int_{\D} \big|f(z)\big|^{p-2}\big|f'(z)\big|^{2}\big|\varphi_a'(\varphi_a(z))\big|^{\alpha+2} (1-|\varphi_a(z)|^2)^{\alpha+2}~\frac{dA(z)}{(1-|z|^2)^2} \\
                        & = \int_{\D} \big|f(z)\big|^{p-2}\big|f'(z)\big|^{2} \left(\frac{1 - |\varphi_a(z)|^2}{|\varphi_a'(z)|}\right)^{\alpha+2
                        }~\frac{dA(z)}{(1-|z|^2)^2} \\
                        & = \int_{\D} \big|f(z)\big|^{p-2}\big|f'(z)\big|^{2} (1-|z|^2)^{\alpha}~dA(z) \\
                        & \simeq \|f\|^p_{\alpha,p},
                    \end{aligned}
                \]
                since 
                \[
                    \frac{1-|\varphi_a(z)|^2}{|\varphi_a'(z)|^2}=1-|z|^2.
                \]
                Hence we are left to show that $I_2$ is controlled by $\|f\|^p_{\alpha,p}$, uniformly in $a$. 
            
            By taking a derivative of the identity
                \[
                    \varphi_a'(\varphi_a(z))=\frac{1}{\varphi_a'(z)}=-\frac{(1-\overline az)^2}{1-|a|^2}
                \]
                we obtain
                \[
                    \varphi_a''(\varphi_a(z))=-\frac{\varphi_a''(z)}{\varphi_a'(z)^3}=-2 \overline{a} \frac{(1-\overline az)^3}{(1-|a|^2)^2}.
                \]
            
            Hence using the conformal invariance of the hyperbolic measure we have that
                \[
                    \begin{aligned}
                        I_2 & = \int_{\D} \big|f(\varphi_a(z))\big|^{p}\big|\varphi_a'(z)\big|^{\alpha - 2}\big| \varphi_a''(z)\big|^{2} (1-|z|^2)^{\alpha+2}~\frac{dA(z)}{(1-|z|^2)^2}. \\
                        & \simeq |a|^2  \int_{\D} \big|f(z)\big|^{p} \left(\frac{1-|\varphi_a(z)|^2}{|\varphi_a'(z)|}\right)^{\alpha+2} ~\frac{dA(z)}{\left|1-\overline az\right|^2(1-|z|^2)^2} \\
                        & = |a|^2 \int_{\D} \big|f(z)\big|^{p} \frac{(1-|z|^2)^{\alpha}}{|1-\overline{a}z|^2}~dA(z)
                    \end{aligned}
                \]
                which is controlled by $\|f\|_{L^p\left(\D, d\mu^{\alpha, 2}_a\right)}^p,$ uniformly in $a$ if and only if $ \{\mu_a^{\alpha, 2} \}_{a\in\D}$ are uniformly Carleson measures for $A^p_\alpha$.
        \end{proof}

        In view of Proposition \ref{prop:CM_same}, Theorem \ref{theo:confinv} is proven once we show that the Carleson constants of $(\mu^{\alpha, 2}_a)_{a\in\D}$ for the embedding $ A^2_\alpha\subset L^2(\D, d\mu^{\alpha, 2}_a)$ are bounded uniformly in $a\in\D$ if $\alpha>0$. More generally, the following holds:

        \begin{proposition}\label{prop:cm_a}
            Let $\alpha=\gamma+2-s>0$ and $s\ge0$. Then there exists a $C>0$ such that
                \[
                    \|f\|_{L^2(\D, d\mu^{\gamma, s}_a)}\leq C~\|f\|_{\alpha,2}
                \]
                for all $f$ in $A^2_\alpha$ and $a\in\D$.
        \end{proposition}

        We deduce Proposition \ref{prop:cm_a} from Theorem \ref{theo:ARS}. In order to do so, we need to estimate $\mu^{\gamma, s}_a(S_w)$, for all $a$ and $w$ in the unit disc:

        \begin{lemma}\label{lemma:mu_boxes}
            Let $\gamma+2-s>0$ and $s \geq 0$. Then
                \[
                    \mu^{\gamma, s}_a(\tilde S_w)\simeq\mu^{\gamma,s}_a(S_w)\simeq\frac{(1-|w|)^{\gamma+2}}{|1-\overline aw|^s}, \quad a, w\in\D.
                \]
                The hidden constants are independent of $a$ and $w$.
        \end{lemma}

        \begin{proof}
            We will make use of the elementary estimate
                \begin{equation}\label{eqn:Nikos_Lemma}
                    |1-\overline az|\simeq\max\{1-|a|, 1-|z|, |\arg(\overline az)|\}, \quad a,z\in\D,
                \end{equation}
                where the hidden constants are universal. In particular, this implies that
                \[
                    \mu^{\gamma, s}_a(\tilde S_w)=\int_{\tilde S_w}\frac{(1-|z|)^\gamma}{|1-\overline az|^s}~dA(z)\simeq\int_{ S_w}\frac{(1-|z|)^\gamma}{|1-\overline az|^s}~dA(z)=\mu^{\gamma, s}_a(S_w).
                \]
                We will now estimate $\mu^{\gamma, s}_a(S_w)$ by distinguishing three cases:
        
            \begin{itemize}
                \item $w$ belongs to $S_a$. In this case, any $z$ in $S_w$ is in a box centered at $a$ with twice the width of $S_a$, thus the maximum in \eqref{eqn:Nikos_Lemma} is comparable to $1-|a|$. Therefore
                    \[
                        \mu^{\gamma, s}(S_w)=\int_{ S_w}\frac{(1-|z|)^\gamma}{|1-\overline az|^s}~dA(z)\simeq\frac{1}{(1-|a|)^s}\int_{ S_w}(1-|z|)^\gamma~dA(z)\simeq\frac{(1-|w|)^{\gamma+2}}{|1-\overline aw|^s}.
                    \]
            
                \item $a$ belongs to $S_w$. This is the case that requires a bit more care. Given $a$ in $\D$ and $n\ge0$, set 
                    \[
                        I^n_a:=\left\{\theta\in\T\,\bigg|\, \frac{|\arg(a\overline{\theta})|}{2\pi}<2^{n-1}(1-|a|)\right\}.
                    \]
                    In particular, $I^0_a$ is the base of the box $S_a$, while $I^n_a$ are the dyadic dilates of $I^0_a$, for all $n\ge1$. Define $a_n$ as the point in $\D$ such that $I^0_{a_n}=I_a^n$.  Finally, let $\eta(a, w)$ denote the smallest positive integer $n$ such that $S_w\subset S_{a_n}$. Note that, by construction, 
                    \begin{equation}
                    \label{eqn:eta}
                        \left|\eta(a, w)-\log_2\frac{1-|w|}{1-|a|}\right|\le 1. 
                    \end{equation}
                    Let $E_n(a):=S_{a_n}\setminus S_{a_{n-1}}$ for all $n\ge1$, $E_0:=S_a$. Hence for all $z$ in $E_n(a)$, the maximum in \eqref{eqn:Nikos_Lemma} is comparable to $2^n(1-|a|)$. Therefore 
                    \[
                        \begin{split}
                            \mu^{\gamma, s}_a(S_w)&=\int_{S_w}\frac{(1-|z|)^\gamma}{|1-\overline az|^s}~dA(z)\\
                            &\simeq\frac{1}{(1-|a|)^s}\sum_{n=0}^{\eta(a, w)}\frac{1}{2^{sn}}\int_{E_n(a)}(1-|z|)^\gamma\,dA(z)\\
                            &=\frac{1}{(1-|a|)^s}\sum_{n=0}^{\eta(a, w)}\frac{(1-|a_n|)^{\gamma+2}-(1-|a_{n-1}|)^{\gamma+2}}{2^{sn}}\\
                            &\simeq(1-|a|)^{\gamma+2-s}\sum_{n=0}^{\eta(a, w)}2^{n(\gamma+2-s)}\\
                            &\simeq(1-|w|)^{\gamma+2-s}\simeq\frac{(1-|w|)^{\gamma+2}}{|1-\overline aw|^s}, 
                        \end{split}
                    \]
                    thanks to \eqref{eqn:eta}, the fact that $\gamma+2-s>0$ and $|1-\overline aw|\simeq 1-|w|$, $a$ being in $S_w$.
            
                \item $a\notin S_w$ and $w\notin S_a$. Hence for all $z$ in $S_w$ the maximum in \eqref{eqn:Nikos_Lemma} is comparable to $|\arg(\overline az)|$, which is independent of $|z|$. Therefore
                \[
                    \begin{split}
                        \mu_a^{\gamma, s}(S_w)&=\int_{S_w}\frac{(1-|z|)^\gamma}{|1-\overline az|^s}~dA(z)\\
                        &\simeq\int_{I^0_{w}}\frac{d\theta}{|\arg(\overline ae^{i\theta})|^s}\int_{1-|w|}^1(1-r)^\gamma\,dr\\
                        &\simeq\frac{(1-|w|)^{\gamma+2}}{|\arg(\overline aw)|^s}\simeq\frac{(1-|w|)^{\gamma+2}}{|1-\overline aw|^s}.
                    \end{split}
                \]
            \end{itemize}
        \end{proof}
        
        We are now ready to prove Proposition \ref{prop:cm_a}.
        
        \begin{proof}[Proof of Proposition \ref{prop:cm_a}]
            Thanks to Theorem \ref{theo:ARS}, it suffices to show that $\mu^{\gamma, s}_a$ satisfies \eqref{eqn:ARS}, where $C$ does not depend on $a$. A few applications of Lemma \ref{lemma:mu_boxes} yield:
                \[
                    \begin{split}
                        &\int_{\tilde S_w}\mu_a^{\gamma, s}(S_z\cap S_w)^2\frac{dA(z)}{(1-|z|)^{\alpha+2}}\\
                        \lesssim&\int_{\tilde S_w}\frac{(1-|z|)^{\gamma+s}}{|1-\overline az|^{2s}}~dA(z)\\
                        =&\mu_a^{\gamma+s, 2s}(\tilde{S}_w)\\
                        \simeq& \mu_a^{\gamma, s}(S_w)\left(\frac{1-|w|^2}{|1-\overline aw|}\right)^s,\\
                        \lesssim & \mu_a^{\gamma, s}(S_w)
                    \end{split}
                \]
            where the last inequality is due to \eqref{eqn:Nikos_Lemma} and $s\ge0$. This concludes the proof.
        \end{proof}

        \begin{remark}
        \label{rem:AM}
            For $p > 0$ and $\alpha > 1$, let 
                \[
                    B^p_\alpha:=\left\{f\in \Hol \,\bigg|\, \int_\D |f'(z)|^p(1-|z|)^{\alpha-2}\,dA(z) < \infty \right\},
                \]
                denote the classic Besov space on the unit disc. If $p < \alpha$, the same argument as in Proposition \ref{prop:conf_index_cm} shows that  $\sup_{a\in\D}\big\|T^{\frac{\alpha}{p}-1}_a\big\|_{\mathcal{B}(B^p_\alpha)}<\infty$ if and only if the measures $\{\mu^{\alpha-2, p}_a\}_{a\in\D}$ are uniformly Carleson for $B^p_\alpha$. This, together with \cite[Thm. 1]{ARS02} and Lemma \ref{lemma:mu_boxes}, gives an alternative proof that the index of conformal invariance of $B^p_\alpha$ is $\frac{\alpha}{p}-1$, \cite{AM21}.
        \end{remark}

\bibliographystyle{abbrv}
\bibliography{Arxiv_bibliography}
        
\end{document}